\documentclass[11pt,reqno]{amsproc}

\title[]{Coercivity of the Dirichlet-to-Neumann operator and applications to the Muskat problem}

\author{Huy Q. Nguyen}
\address{
Department of Mathematics\\
University of Maryland\\
College Park, MD 20742, USA
}
\email[H. Nguyen]{hnguye90@umd.edu}

\usepackage[margin=1in]{geometry}
\usepackage{amsmath, amsthm, amssymb, mathrsfs, stmaryrd}
\usepackage{times}
\usepackage{color}
\usepackage{hyperref}

\newcommand{\bq}{\begin{equation}}
\newcommand{\eq}{\end{equation}}
\newcommand{\bqa}{\begin{eqnarray*}}
\newcommand{\eqa}{\end{eqnarray*}}
\theoremstyle{plain}
\newtheorem{theo}{Theorem}[section]
\newtheorem{prop}[theo]{Proposition}
\newtheorem{lemm}[theo]{Lemma}
\newtheorem{coro}[theo]{Corollary}

\theoremstyle{definition}

\DeclareMathOperator{\cnx}{div}

\DeclareMathOperator{\Lip}{Lip}
\DeclareSymbolFont{pletters}{OT1}{cmr}{m}{sl}
\DeclareMathSymbol{s}{\mathalpha}{pletters}{`s}

\def\eps{\varepsilon}
\def\na{\nabla}

\def\mez{\frac{1}{2}}

\def\Rr{\mathbb{R}}
\def\T{\mathbb{T}}
\def\Nn{\mathbb{N}}

\def\p{\partial}
\def\na{\nabla}

\def\wt{\widetilde}

\def\Ll{L_{\text{loc}}}
\def\wh{\widehat}

\numberwithin{equation}{section}

\date{today}
\begin{document}
\begin{abstract}
We consider the Dirichlet-to-Neumann operator in strip-like and half-space domains with Lipschitz boundary. It is shown that the  quadratic form generated by the Dirichlet-to-Neumann operator  controls some sharp homogeneous  fractional Sobolev norms. As an application, we prove that the global Lipschitz solutions constructed in \cite{DGN} for the one-phase Muskat problem decays exponentially in time in any H\"older norm $C^\alpha$, $\alpha\in (0, 1)$. 
\end{abstract}
\noindent

\maketitle
\begin{center}{\em To Professor Duong Minh Duc, with gratitude, respect and admiration}\end{center}
\section{Introduction}
The Dirichlet-to-Neumann operator  arises naturally in free boundary problems in fluid mechanics as a result of dimension reduction. To name a few, the water wave, the Muskat and the Hele-Shaw problem \cite{ABZ3,  AMS, ChlGuiSch, NgPa}.  

Let  $M$ be either the real line $\Rr$ or  the circle $\T$. We consider either the strip-like domain
\bq\label{striplike}
\Omega=\{(x, y) \in M^d\times \Rr: b(x)< y<f(x)\} 
\eq
or the half-space 
\bq\label{halfspace}
\Omega=\{(x, y)\in M^d\times \Rr: y<f(x)\}.
\eq
The boundary functions $f$ and $b$ are Lipschitz continuous and  satisfy
\bq\label{separation} 
\inf_{x\in M^d} f(x)-b(x)\ge h>0. 
\eq 
We also refer to \eqref{striplike} as the finite depth case and to \eqref{halfspace} as the infinite depth case. 

Given a function $g: M^d\to \Rr$, $d\ge 1$, we consider the boundary value problem
\bq\label{sys:elliptic}
\begin{cases}
\Delta_{x, y}\phi=0\quad\text{in } \Omega,\\
\phi(x, f(x))=g(x),\\
\p_\nu\phi(x, b(x))=0, 
\end{cases}\
\eq
where $\nu=\frac{1}{\sqrt{|\na_xb|^2+1}}(\na_x b, -1)$ is the outward unit normal to the bottom boundary $\{y=b(x)\}$. In the infinite depth case, the Neumann condition in \eqref{sys:elliptic} is replaced by the decay condition
\bq
\lim_{(x, y)\to \infty}\na_{x, y}\phi=0.
\eq
The Dirichlet-Neumann operator $G$ associated to $\Omega$ is defined by
\bq
G(g)=(\p_y\phi-\na_xf\cdot \na_x \phi)\vert_{y=f(x)}=(-\na_x f, 1)\cdot \na_{x, y}\phi\vert_{y=f(x)}.
\eq
In other words, $G(g)$ is the normal derivative of the harmonic function $\phi$ on the top boundary $\{y=f(x)\}$.

For the perfect half-space, i.e. $f=0$, we have $G(g)=|D|g$, where $|D|$ is the Fourier multiplier $|\xi|$. In other words, $|D|$ is the square root of the Laplacian $-\Delta_x$. %%We have the following well-known C\'ordoba-C\'ordoba inequality \cite{CC0, CC}
%\bq\label{C-C}
%g|D|g\ge \mez |D|g^2. 
%\eq
%This inequality is most useful as a pointwise inequality since its integrated version is trivial:
%\[
%\int_{M^d} g|D|g\ge \mez\int_{M^d} |D|g^2=0.
%\]
%On the other hand, we have
The quadratic form generated by $|D|$ is coercive:
\bq\label{C-C:i}
\int_{M^d} g|D|g=\| |D|^\mez g\|_{L^2(M^d)}^2=\| g\|_{\dot H^\mez(M^d)}^2.
\eq
%It was recently obtained in \cite{AMS} that
%\bq\label{Alazard}
%gG(g)\ge \mez G(g^2)
%\eq
%for the half-space domain \eqref{halfspace} with periodic boundary $f$. In fact,  \eqref{Alazard} also holds for the strip-like domain \eqref{striplike} with the same proof. 
On the other hand, for straight strip domains, i.e. $f=0$ and $b(x)\equiv -a$ with $a>0$, we have $G(g)=|D|\tanh(a|D|)$, whence
\bq\label{C-C:i2}
\int_{M^d} g|D|\tanh(a|D|)g=\| [|D|\tanh(a|D|)]^\mez g\|_{M^d}^2,
\eq
where the right hand-side is equivalent to the seminorm $\dot H^\mez(\T^d)$ when $M^d=\T^d$ and to the seminorm
\bq\label{def:wtH}
\| g\|_{\wt H^\mez(\Rr^d)}^2:=\int_{M^d}\min\{|\xi|, |\xi|^2\}|\wh{g}(\xi)|^2d\xi
\eq
when $M^d=\Rr^d$. The Sobolev type space $\wt H^s$ is studied in detail in \cite{LeoniTice}. We also refer to \cite{CC0, CC, AMS} for pointwise lower bounds for $g|D|g$ and $gG(g)$. 

With applications to free boundary problems in mind, we are interested in generalizing  \eqref{C-C:i} and \eqref{C-C:i2} to non flat boundary, i.e. to domains of the form \eqref{striplike} and \eqref{halfspace} with nontrivial boundary functions $f$ and $g$.

 It is known that when $g$ belongs to the fractional Sobolev space $H^\mez(M^d)$, $G(g)$ is well-defined in $H^{-\mez}(M^d)$. See Proposition \ref{prop:DN} below. We shall prove the following coercive inequalities that generalize \eqref{C-C:i} and \eqref{C-C:i2} to the domain  \eqref{halfspace} and \eqref{striplike}     respectively:
 \bq\label{intro:est1}
 \langle G(g), g\rangle_{H^{-\mez}(M^d), H^\mez(M^d)}\ge M\| g\|_{X}^2,
 \eq
 where $X$ is either $\wt H^\mez$ or $\dot H^\mez$ depending on $M$ and the depth of $\Omega$; the constant $M$  depends explicitly on the boundary of $\Omega$. See Propositions \ref{prop:coercive:finite} and  \ref{prop:coercive:inf} below. 
 
 In Proposition \ref{prop:Phi} we establish the  coercive inequality
 \bq\label{intro:est2}
  \langle G(g), \Phi'(g)\rangle_{H^{-\mez}(M^d), H^\mez(M^d)}\ge M\| \Psi(g)\|^2_{X},
 \eq
 where $\Phi$ is any $C^2$ convex function such that $\Phi'(z)/z$ is continuous, and $\Psi(z)=\int_0^z \sqrt{\Phi''(z')}dz'$. As a consequence, when $M=\T$ and $g$ has zero mean, $  \langle G(g), \Phi'(g)\rangle_{H^{-\mez}(M^d), H^\mez(M^d)}$ controls the $L^p$ norm of $g$. 
 
 In section \ref{section:Muskat}, we apply \eqref{intro:est1} to obtain time decay of the global Lipschitz solutions constructed in \cite{DGN} for the one-phase Muskat problem. It is shown that for any data $f_0\in W^{1, \infty}(\T)$, the global solution $f$ satisfies
 \[
 f\in L^2((0, \infty); \dot H^\mez(\T)),\quad  \p_tf\in L^2((0, \infty); \dot H^{-\mez}(\T)).
 \]
 If $f_0$ has zero mean, we prove that all the H\"older norms $C^\alpha(\T)$, $\alpha \in (0, 1)$ of $f$ decay exponentially.
\section{Coercive inequalities for the Dirichlet-to-Neumann operator}
We denote
\[
\Lip(M^d)=\left\{u: M^d\to \Rr:~\exists C>0,~\forall x, x'\in M^d,~|u(x)-u(x')|\le C|x-x'|\right\}.
\]
We first recall the following proposition on the boundedness of the Dirichlet-to-Neumann operator. 
\begin{prop}[\protect{\cite{ABZ3, NgPa}}]\label{prop:DN}
Let $d\ge 1$.

1. (The finite depth case) Assume that $b$, $f\in \Lip(M^d)$ such that $f-b\in L^\infty(M^d)$ and \eqref{separation} holds. Let $\widetilde{H}^\mez(\Rr^d)$ be the space of $L^2_{\text{loc}}(\Rr^d)$ functions whose Fourier transform are locally $L^2$ in the complement of the origin such that the seminorm \eqref{def:wtH}
%\[
%\| g\|_{\widetilde{H}^\mez(\Rr^d)}^2= \int_{\Rr^d}\min\{|\xi|, |\xi|^2\}|\wh{g}(\xi)|^2d\xi
%\]
is finite. For notational convenience, we set $\widetilde{H}^\mez(\T^d)=\dot H^\mez(\T^d)$. 

For any $g\in \widetilde{H}^\mez(M^d)$, there exists a unique solution $\phi\in \dot H^1(\Omega)$ to \eqref{sys:elliptic} and we have $G(g)\in H^{-\mez}(M^d)$ together with the bound
\bq\label{bound:DNfinite}
\| G(g)\|_{H^{-\mez}(M^d)}\le C(\| \na f\|_{L^\infty(M^d)}+\| \na b\|_{L^\infty(M^d)})\| g\|_{\widetilde{H}^\mez(M^d)},
\eq
where $C=C(h, d)$.

2. (The infinite depth cases) Let $f\in \Lip(M^d)$. For any $g\in \dot H^\mez(M^d)$, there exists a unique solution $\phi\in \dot H^1(\Omega)$ to \eqref{sys:elliptic} and we have $G(g)\in H^{-\mez}(M^d)$ together with the bound
\bq\label{bound:DNinf}
\| G(g)\|_{H^{-\mez}(M^d)}\le C(\| \na f\|_{L^\infty(M^d)})\| g\|_{\dot H^\mez(M^d)},
\eq
where $C=C(d)$. 
\end{prop}
Coercive inequalities for $\langle G(g), g\rangle_{H^{-\mez}(M^d), H^\mez(M^d)}$ are established in Propositions \ref{prop:coercive:finite} and  \ref{prop:coercive:inf} for the finite and infinite depth cases respectively.
\begin{prop}\label{prop:coercive:finite} Let $\Omega$ be the strip-like domain \eqref{striplike}, where  $b$, $f\in \Lip(M^d)$ such that $f-b\in L^\infty(M^d)$ and \eqref{separation} holds.  There exists a constant $C=C(d)>0$ such that for any $g\in H^\mez(M^d)$,  we have
\bq\label{coercive:finite}
\langle G(g), g\rangle_{H^{-\mez}(M^d), H^\mez(M^d)}\ge \frac{Ch}{1+\|\na f\|_{L^\infty}^2+\|f-b\|_{W^{1, \infty}}^2}\| g\|_{\widetilde H^\mez(M^d)}^2,
\eq
where $h$ is given by \eqref{separation}. 
\end{prop}
\begin{proof}
We flatten $\Omega$ using the Lipschitz diffeomorphism 
\[
M^d\times (-1, 0)\ni (x, z)\mapsto \mathcal{S}(x, z)=(x, \varrho(x, z))\in \Omega,
\]
where
\[
\varrho(x, z)=(z+1)f(x)-zb(x)
\]
satisfies $\p_z\varrho(x, z)=f(x)-b(x)\ge h$ and $\na_{x, z}\varrho\in L^\infty(M^d\times (-1, 0))$. By the chain rule, the function $v=\phi\circ \mathcal{S} $  satisfies 
\bq\label{div:eq}
\cnx_{x, z}(\mathcal{A}\na_{x,z}v)(x, z)=\p_z\varrho (\Delta_{x, y}\phi)(\mathcal{S}(x, z))=0,
\eq
where 
\bq\label{def:matrixA}
\mathcal{A}=
\begin{bmatrix}
\p_z\varrho\mathbb{I}_{d\times d} & -\na_x\varrho\\
-(\na_x\varrho )^T& \frac{1+|\na_x\varrho|^2}{\p_z\varrho}
\end{bmatrix}.
\eq
Here we regard the gradient as a column matrix.  In terms of $v$ we have 
\bq\label{DN:v}
G(g)(x)=-\na_x \varrho(x, 0)\cdot \na_xv(x, 0)+\frac{1+|\na_x \varrho(x, 0)|^2}{\p_z\varrho(x, 0)}\p_zv(x, 0)=e_{d+1}\cdot (\mathcal{A}\na_{x, z}v)(x, 0). 
\eq
We recall the following Stokes formula: 
\bq\label{Stokes}
\begin{aligned}
&\langle e_{d+1}\cdot u (\cdot, 0), w(\cdot, 0)\rangle_{H^{-\mez}(M^d), H^\mez(M^d)} -\langle e_{d+1}\cdot u (\cdot, -a), w(\cdot, -a)\rangle_{H^{-\mez}(M^d), H^\mez(M^d)} \\
&\quad=(u, \na_{x, z} w)_{L^2(M^d\times (-a, 0))}+(\cnx_{x, z}  u, w)_{L^2(M^d\times (-a, 0))},\quad a>0,
\end{aligned}
\eq
provided that $u\in L^2(M^d\times (-1, 0))^{d+1}$, $\cnx_{x, z} u\in  L^2(M^d\times (-1, 0))$ and $w\in H^1(M^d\times (-1, 0))$.

We check that  \eqref{Stokes} is applicable with $u=\mathcal{A}\na_{x, z}v$ and $w=v$. Indeed, since $\na_{x, y}\phi\in L^2(\Omega)$ (by Proposition \ref{prop:DN}) and $\na_{x, z}\varrho\in L^\infty(M^d\times (-1, 0))$, we have $\na_{x, z}v\in L^2(M^d\times (-1, 0))$, and thus  $\mathcal{A}\na_{x, z}v\in  L^2(M^d\times (-1, 0))$. In addition, since $v(\cdot, 0)=g(\cdot)\in L^2(M^d)$ and $\Omega$ has finite depth, it follows that $v\in L^2(M^d\times (-1, 0))$. By the chain rule, we have
\bq\label{Neumann:v}
\begin{aligned}
e_{d+1}\cdot (\mathcal{A}\na_{x, z}v)\vert_{z=-1}&=-\na_x\varrho\cdot \na_xv+\frac{1+|\na_x\varrho|^2}{\p_z\varrho}\p_zv\vert_{z=-1}\\
&=-\na_x\varrho\cdot \na_x\phi+\p_y\phi\vert_{z=-1}\\
&=-\na_xb\cdot \na_x\phi+\p_y\phi\vert_{z=-1}\\
&=-\sqrt{1+|\na_x b|^2}\p_\nu \phi(x, b(x))=0.
\end{aligned}
\eq
Then applying \eqref{Stokes} and invoking \eqref{div:eq}, \eqref{DN:v} and \eqref{Neumann:v}, we deduce 
\bq\label{DN:videntity}
\begin{aligned}
\langle G(g), g\rangle_{H^{-\mez}, H^\mez}&=\int_{-1}^0\int_{M^d} \mathcal{A}\na_{x, z}v\cdot \na_{x,z }vdxdz\\
&=\int_{-1}^0\int_{M^d}\p_z\varrho \left\{|\na_xv|^2-2\frac{\na_x\varrho}{\p_z\varrho}\cdot \na_xv\p_zv+\frac{1+|\na_x\varrho|^2}{|\p_z\varrho|^2}|\p_zv|^2 \right\}dxdz\\
&=\int_{-1}^0\int_{M^d}\p_z\varrho \left\{\left|\na_xv-\frac{\na_x\varrho}{\p_z\varrho}\p_zv\right|^2+\frac{|\p_zv|^2 }{|\p_z\varrho|^2}\right\}dxdz.
\end{aligned}
\eq
In the remainder of this proof, we only treat the more difficult case $M^d=\Rr^d$. Let $\chi: \Rr\to \Rr$ be a smooth function that is identically $1$ on $(-1/3, \infty)$ and vanishes on $(-\infty, -2/3)$. Then $w(x, z):=\chi(z)v(x, z)$ satisfies $w(x, 0)=g(x)$ and $w$ vanishes near $z=-1$. Consequently,
\begin{align*}
|\wh{g}(\xi)|^2=|\wh{w}(\xi, 0)|^2&=\Re\int_{-1}^0\p_z\wh{w}(\xi, z)\overline{\wh{w}(\xi, z)}dz\\
&=\Re\int_{-1}^0\left[\chi'(z)\wh{v}(\xi, z)+\chi(z)\p_z\wh{v}(\xi, z)\right]\chi(z)\overline{\wh{v}(\xi, z)}dz,
\end{align*}
where $\wh{w}$ is the Fourier transform of $w$ with respect to $x\in \Rr^d$.  It follows that
\begin{align*}
\int_{\Rr^d}\min\{|\xi|, |\xi|^2\}|\wh{g}(\xi)|^2%&=\Re\int_{-1}^0\int_{\Rr^d}\chi(z)\chi'(z)\min\{|\xi|, |\xi|^2\}||\wh{v}(\xi, z)|^2\\
%&\qquad+\chi^2(z)\p_z\wh{v}(\xi, z)\min\{|\xi|, |\xi|^2\}|\overline{\wh{v}(\xi, z)}d\xi dz\\
&\le C\int_{-1}^0\int_{\Rr^d}|\xi|^2|\wh{v}(\xi, z)|^2+|\p_z\wh{v}(\xi, z)||\xi||\wh{v}(\xi, z)|d\xi dz\\
&\le C\int_{-1}^0\int_{\Rr^d}|\wh{\na_xv}(\xi, z)|^2+|\p_z\wh{v}(\xi, z)||\wh{\na_xv}(\xi, z)|d\xi dz\\
&\le C\| \na_xv\|_{L^2(\Rr^d\times (-1, 0))}^2+C\| \na_xv\|_{L^2(\Rr^d\times (-1, 0))}\| \p_zv\|_{L^2(\Rr^d\times (-1, 0))}\\
&\le C\int_{-1}^0\int_{\Rr^d} |\na_x v|^2+|\p_zv|^2dxdz.
\end{align*}
It follows from this and the triangle inequality
\[
|\na_x v|\le |\na_x v-\frac{\na_x\varrho}{\p_z\varrho}\p_zv|+\frac{|\na_x\varrho|}{\p_z\varrho}|\p_zv|
\]
that
\[
\begin{aligned}
\int_{\Rr^d}\min\{|\xi|, |\xi|^2\}||\wh{g}(\xi)|^2&\le C\int_{-1}^0\int_{\Rr^d}|\na_x v-\frac{\na_x\varrho}{\p_z\varrho}\p_zv|^2+\big(|\p_z\varrho|^2+|\na_x\varrho|^2\big)\frac{|\p_zv|^2}{|\p_z\varrho|^2}dxdz\\
&\le C\int_{-1}^0\int_{\Rr^d}\p_z\varrho \left\{\left|\na_xv-\frac{\na_x\varrho}{\p_z\varrho}\p_zv\right|^2+\frac{|\p_zv|^2 }{|\p_z\varrho|^2}\right\}\frac{1+|\p_z\varrho|^2+|\na_x\varrho|^2}{\p_z\varrho}dxdz.
\end{aligned}
\]
Using  
\[
h\le \p_z\varrho=f(x)-b(x)\le \| f-b\|_{L^\infty}\quad\text{and}\quad \| \na_x\varrho\|_{L^\infty}\le \| \na f\|_{L^\infty}+\| \na(f-b)\|_{L^\infty},
\]
we deduce
\bq\label{bound:trace}
\begin{aligned}
&\int_{\Rr^d}\min\{|\xi|, |\xi|^2\}|\wh{g}(\xi)|^2\\
&\quad\le C\frac{1+\|\na f\|_{L^\infty}^2+\|f-b\|_{W^{1, \infty}}^2}{h}\int_{-1}^0\int_{\Rr^d}\p_z\varrho \left\{\left|\na_xv-\frac{\na_x\varrho}{\p_z\varrho}\p_zv\right|^2+\frac{|\p_zv|^2 }{|\p_z\varrho|^2}\right\}dxdz.
\end{aligned}
\eq
In view of \eqref{DN:videntity} and \eqref{bound:trace}  we conclude the proof of \eqref{coercive:finite}.
\end{proof} 
\begin{prop}\label{prop:coercive:inf}
Let $\Omega$ be the half-space domain \eqref{halfspace} with $f\in \Lip(M^d)$. There exists a constant $C=C(d)>0$ such that for any  $g\in H^\mez(M^d)$, we have
\bq\label{coercive:inf}
\langle G(g), g\rangle_{H^{-\mez}(M^d), H^\mez(M^d)}\ge \frac{C}{1+\| \na f\|_{L^\infty(M^d)}}\|g\|_{\dot H^\mez(M^d)}^2. 
\eq
\end{prop}
\begin{proof}
We flatten $\Omega=\{(x, y)\in M^d\times \Rr: y<f(x)\}$ using the Lipschitz diffeomorphism 
\[
M^d\times  (-\infty, 0)\ni (x, z)\mapsto \mathcal{S}(x, z)=(x, \varrho(x, z))\in \Omega,
\]
where $\varrho(x, z)=z+f(x)$. The formula \eqref{DN:v} holds with $v=\phi\circ \mathcal{S}$. Let $\chi:\Rr\to \Rr$ be a smooth function satisfying $\chi(z)=1$ on $(-1/3, \infty)$ and $\chi(z)=0$ on $(-\infty, -2/3)$. We apply the Stokes formula \eqref{Stokes}  with $u=\mathcal{A}\na_{x, z}v$, $w=v(x, z)\chi(\frac{z}{-n})$ and $a=n$ to have
\bq\label{dual:inf}
\langle G(g), g\rangle_{H^{-\mez}, H^\mez}=\int_{M^d} \int_{-n}^0\mathcal{A}\na_{x, z}v\cdot \na_{x,z }v-\frac{1}{n}\big(e_{d+1}\cdot\mathcal{A}\na_{x, z}v\big)\chi'(\frac{z}{-n}) vdzdx.
\eq
We shall prove that 
\bq\label{zerolimit}
I:=\lim_{n\to \infty}\frac{1}{n}\int_{M^d} \int_{-n}^0 \big(e_{d+1}\cdot\mathcal{A}\na_{x, z}v\big)\chi'(\frac{z}{-n}) vdzdx=0.
\eq
Since $v(x, 0)=g(x)$, we have
\[
|v(x, z)|\le |g(x)|+|z|^\mez\left|\int_{z}^0|\p_zv(x, z')|^2dz'\right|^\mez\le |g(x)|+n^\mez\left|\int_{-n}^0|\p_zv(x, z')|^2dz'\right|^\mez,\quad z\in [-n, 0],
\]
whence
\begin{align*}
I&\le \frac{1}{n}\int_{M^d} \int_{-n}^0 \big|e_{d+1}\cdot\mathcal{A}\na_{x, z}v\big||\chi'(\frac{z}{-n})||g(x)|dzdx\\
&\quad+ \frac{1}{\sqrt{n}}\int_{M^d} \int_{-n}^0 \big|e_{d+1}\cdot\mathcal{A}\na_{x, z}v\big||\chi'(\frac{z}{-n})|\left|\int_{-n}^0|\p_zv(x, z')|^2dz'\right|^\mez dzdx:=I_1+I_2.
\end{align*}
By H\"older's inequality,
\begin{align*}
%I_1&\le \frac{1}{n}\left(\int_{M^d}\int_{-n}^0\big|\mathcal{A}\na_{x, z}v(x, z)\big|^2dzdx\right)^\mez\left(\int_{-n}^0|\chi'(\frac{z}{-n})|^2dz\right)^\mez\\
%&\le \frac{C}{n^\mez}\left(\int_{M^d}\int_{-n}^0\big|\mathcal{A}\na_{x, z}v(x, z)\big|^2dzdx\right)^\mez\to 0\quad\text{as } n\to \infty,
I_1&\le \frac{1}{n}\|\mathcal{A}\na_{x, z}v\|_{L^2(M^d\times (-n, 0))}\|g\chi'(\frac{\cdot}{-n})\|_{L^2(M^d\times (-n, 0))}\\
&\le \frac{C}{n^\mez}\|\mathcal{A}\na_{x, z}v\|_{L^2(M^d\times (-n, 0))}\| g\|_{L^2(M^d)}\to 0\quad\text{as } n\to \infty
\end{align*}
and
\begin{align*}
%I_2&\le  \frac{1}{n} \left(\int_{M^d}\int_{-n}^0 \big|e_{d+1}\cdot\mathcal{A}\na_{x, z}v\big|^2|\chi'(\frac{z}{-n})|^2dzdx\right)^\mez \left(\int_{M^d}\int_{-n}^0|z|\int_{z}^0|\p_zv(x, z')|^2dz'dzdx\right)^\mez \\
%&\le   \frac{1}{n} \left(\int_{M^d}\int_{-n}^0 \big|e_{d+1}\cdot\mathcal{A}\na_{x, z}v\big|^2|\chi'(\frac{z}{-n})|^2dzdx\right)^\mez \left(\int_{M^d}\int_{-n}^0n\int_{-n}^0|\p_zv(x, z')|^2dz'dzdx\right)^\mez\\
%&=  \left(\int_{M^d}\int_{-n}^0 \big|e_{d+1}\cdot\mathcal{A}\na_{x, z}v\big|^2|\chi'(\frac{z}{-n})|^2dzdx\right)^\mez \left(\int_{M^d}\int_{-n}^0|\p_zv(x, z')|^2dz'dx\right)^\mez 
I_2&\le  \frac{1}{\sqrt{n}} \|\mathcal{A}\na_{x, z}v\chi'(\frac{\cdot}{-n})\|_{L^2(M^d\times (-n, 0))}\| \int_{-n}^0\left|\p_zv(\cdot, z')|^2dz'\right|^\mez\|_{L^2(M^d\times (-n, 0))} \\
&\le  \|\mathcal{A}\na_{x, z}v\chi'(\frac{\cdot}{-n})\|_{L^2(M^d\times (-n, 0))}\|\p_zv\|_{L^2(M^d\times (-n, 0))}.
 \end{align*}
 Since $\chi'(\frac{z}{-n})\to 0$ as $n\to \infty$ and $\mathcal{A}\na_{x, z}v \in L^2(M^d\times (-\infty, 0))$, the dominated convergence theorem implies that $\lim_{n\to \infty} I_2=0$. Therefore, passing $n\to \infty$ in \eqref{dual:inf} we obtain
 \bq\label{DN:videntity:inf}
 \begin{aligned}
\langle G(g), g\rangle_{H^{-\mez}, H^\mez}&=\int_{M^d} \int_{-\infty}^0\mathcal{A}\na_{x, z}v\cdot \na_{x,z }vdzdx\\
&=\int_{M^d}\int_{-\infty}^0\left|\na_xv-\na f\p_zv\right|^2+|\p_zv|^2 dzdx,
\end{aligned}
\eq
where we have used that $\na_x\varrho=\na_xf$ and $\p_z\varrho=1$. 

We only consider the more difficult case $M^d=\Rr^d$ in the remainder of this proof. For $w(x, z)=\chi(\frac{z}{-n})v(x, z)$, we have $w(x, 0)=v(x, 0)=g(x)$ and $w$ vanishes near $z=-n$. Consequently,
\begin{align*}
\int_{\Rr^d}|\xi||\wh{g}(\xi)|^2d\xi&=\int_{\Rr^d}|\xi|\int_{-n}^0\p_z|\wh{w}(\xi, z)|^2=\Re\int_{\Rr^d}\int_{-n}^0\p_z\wh{w}(\xi, z)|\xi|\overline{\wh{w}(\xi, z)}dzd\xi \\
&=\Re\int_{\Rr^d}\int_{-n}^0\chi^2(\frac{z}{-n})\p_z\wh{v}(\xi, z)\overline{\wh{|D|v}(\xi, z)}dzd\xi\\
&\quad-\frac{1}{n}\Re\int_{\Rr^d}\int_{-n}^0\chi(\frac{z}{-n})\chi'(\frac{z}{-n})\wh{v}(\xi, z)\overline{\wh{|D|v}(\xi, z)}dzd\xi\\
&=2\int_{\Rr^d}\int_{-n}^0\chi^2(\frac{z}{-n})\p_zv(x, z)|D|v(x, z)dzdx\\
&\quad-\frac{2}{n}\int_{\Rr^d}\int_{-n}^0\chi(\frac{z}{-n})\chi'(\frac{z}{-n})v(x, z)|D|v(x, z)dzdx.
\end{align*}
Since $\p_zv$ and $|D|v$ belong to $L^2(\Rr^d\times \Rr_-)$, arguing as in \eqref{zerolimit}, we can pass to the limit $n\to \infty$ and obtain 
\bq\begin{aligned}
\int_{\Rr^d}|\xi||\wh{g}(\xi)|^2d\xi&=2\int_{\Rr^d}\int_{-\infty}^0\p_zv(x, z)|D|v(x, z)dzdx\\
&=2\int_{\Rr^d}\int_{-\infty}^0\p_zv(x, z)\mathcal{R}\cdot \na_xv(x, z)dzdx\\
&=\int_{\Rr^d}\int_{-\infty}^02\p_zv\mathcal{R}\cdot (\na_xv-\na f\p_zv)+ 2\p_zv\mathcal{R}\cdot(\na f\p_zv)dzdx\\
&=\int_{\Rr^d}\int_{-\infty}^0|\p_zv|^2+|\mathcal{R}\cdot (\na_xv-\na f\p_zv)|^2+ 2\p_zv\mathcal{R}\cdot(\na f\p_zv)\\
&\qquad\qquad -\left[\p_zv-\mathcal{R}\cdot (\na_xv-\na f\p_zv)\right]^2dzdx,
\end{aligned}
\eq
where $\mathcal{R}$ denotes the Riesz transform, $\wh{\mathcal{R}u}(\xi)=\frac{-i\xi}{|\xi|}\wh{u}(\xi)$. Using H\"older's inequality and the boundedness of $\mathcal{R}$ in $L^2$, we obtain
\bq\label{bound:trace2}
\begin{aligned}
\int_{\Rr^d}|\xi||\wh{g}(\xi)|^2d\xi&\le C\| \p_zv\|_{L^2}\| \na_xv-\na f\p_zv\|_{L^2}+C\|\na f\|_{L^\infty}\|\p_zv\|_{L^2}^2\\
&\le C(1+\| \na f\|_{L^\infty})\int_{\Rr^d}\int_{-\infty}^0\left|\na_xv-\na f\p_zv\right|^2+|\p_zv|^2 dzdx.
\end{aligned}
\eq
Finally, \eqref{coercive:inf} follows from \eqref{DN:videntity:inf} and \eqref{bound:trace2}.
\end{proof}
Next, we generalize \eqref{coercive:finite} and \eqref{coercive:inf} to the pairing $\langle G(g), \Phi'(g)\rangle_{H^{-\mez}(M^d), H^\mez(M^d)}$ for convex functions $\Phi$.  
\begin{prop}\label{prop:Phi}
Let $\Phi: \Rr\to \Rr$ be a $C^2$ convex function such that $\Phi'(z)/z$ is continuous on $\Rr$. Set 
\bq\label{def:Psi}
\Psi(z)=\int_0^z \sqrt{\Phi''(z')}dz'.
\eq
Let $g\in H^\mez(M^d)\cap L^\infty(M^d)$.

1) (The finite depth case) If $b$, $f\in \Lip(M^d)$ such that $f-b\in L^\infty(M^d)$ and \eqref{separation} holds, then there exists a constant $C=C(d)>0$ such that
\bq
\langle G(g), \Phi'(g)\rangle_{H^{-\mez}(M^d), H^\mez(M^d)}\ge \frac{Ch}{1+\|\na f\|_{L^\infty}^2+\|f-b\|_{W^{1, \infty}}^2}\| \Psi(g)\|^2_{\wt H^\mez(M^d)}.
\eq
2) (The infinite depth case)  If $f\in \Lip(M^d)$, then there exists a constant $C=C(d)>0$ such that
\bq
\langle G(g), \Phi'(g)\rangle_{H^{-\mez}(M^d), H^\mez(M^d)}\ge \frac{C}{1+\|\na f\|_{L^\infty}}\| \Psi(g)\|^2_{\dot H^\mez(M^d)}.
\eq
\end{prop}
\begin{proof}
We shall only consider the more difficult case $M^d=\Rr^d$. Since  $\Phi'(z)/z$ is continuous and $g\in H^\mez(\Rr^d)\cap L^\infty(\Rr^d)$, it can be shown that $\Phi'(g)\in H^\mez(\Rr^d)\subset \dot H^\mez(\Rr^d)\subset \wt H^\mez(\Rr^d)$.  Let $v=\phi\circ \mathcal{S}$ as given in the proof of Propositions \ref{prop:coercive:finite} and \ref{prop:coercive:inf}. By the maximum principle for the harmonic function $\phi$, we have 
\bq\label{max:v}
\| v\|_{L^\infty(M^d\times J)}= \| \phi\|_{L^\infty(\Omega)}\le \| g\|_{L^\infty(M^d)},
\eq
where $J=(-1, 0)$ in the finite depth case and $J=(-\infty, 0)$ in the infinite depth case. From \eqref{max:v} and the assumption that $\Phi'(z)/z$ is continuous, we deduce that $\Phi'(v)\in L^2(M^d\times J)$. 

1) The finite depth case.  Lemma \ref{lemm:chainrule} below implies that $\na_{x, z}\Phi'(v)=\Phi''(v)\na_{x,z}v\in L^2(M^d\times (-1, 0))$.  Thus we can apply the Stokes formula \eqref{Stokes} with $u=\mathcal{A}\na_{x, z}v$ and $w=\Phi'(v)\in H^1(M^d\times (-1, 0))$ to have
\[
\begin{aligned}
\langle G(g), \Phi'(g)\rangle_{H^{-\mez}, H^\mez}&=\int_{-1}^0\int_{M^d} \mathcal{A}\na_{x, z}v\cdot \na_{x,z }v \Phi''(v)dxdz\\
&=\int_{-1}^0\int_{M^d} \mathcal{A}\na_{x, z}\Psi(v)\cdot \na_{x,z }\Psi(v)dxdz\\
&=\int_{-1}^0\int_{M^d}\p_z\varrho \left\{\left|\na_x\Psi(v)-\frac{\na_x\varrho}{\p_z\varrho}\p_z\Psi(v)\right|^2+\frac{|\p_z\Psi(v)|^2 }{|\p_z\varrho|^2}\right\}dxdz.
\end{aligned}
\]
We then  conclude by following the proof of \eqref{bound:trace} with $\Psi(g)$ in place of $g$ and $\Psi(v)$ in place of $v$.

2) The infinite depth case. The proof proceeds similarly to that of Proposition \ref{prop:coercive:inf} and the finite depth case 1) above. We only remark that in place of \eqref{zerolimit}, we need to prove 
\[
\lim_{n\to \infty} \frac{1}{n}\int_{M^d} \int_{-n}^0 \big(e_{d+1}\cdot\mathcal{A}\na_{x, z}v\big)\chi'(\frac{z}{-n}) \Phi'(v)dxdz=0.
\]
Since $\Phi'(z)/z$ is continuous and $v$ is bounded, we can replace $\Phi'(v)$ by $v$ in the preceding limit and argue as in  the proof of \eqref{zerolimit}.
\end{proof}
\begin{coro}
For any $p\ge 2$, there exist positive constants $C=C(d)$ and $C'=C'(p, d)$ such that for any $g\in H^\mez(\T^d)\cap L^\infty(\T^d)$ satisfying $\int_{\T^d}g=0$, we have
\bq
\langle G(g), p|g|^{p-2}g\rangle_{H^{-\mez}(\T^d), H^\mez(\T^d)}\ge M\left(\| |g|^{p/2-1}g\|^2_{\dot H^\mez(\T^d)}+C'\| g\|^p_{L^p(\T^d)}\right),
\eq
where 
\bq\label{def:M}
M=
\begin{cases}
\frac{Ch}{1+\|\na f\|_{L^\infty}^2+\|f-b\|_{W^{1, \infty}}^2}\quad\text{in~the~finite~depth~case},\\
\frac{C}{1+\| \na f\|_{L^\infty(M^d)}}\quad\text{in~the~infinite~depth~case}.
\end{cases}
\eq
\end{coro}
\begin{proof}
For $p\ge 2$, Proposition \ref{prop:Phi} is applicable with $\Phi(z)=|z|^p$ and $\Psi(z)=2\sqrt{\frac{p-1}{p}}|z|^{p/2-1}z$.  We obtain 
\[
\langle G(g), p|g|^{p-2}g\rangle_{H^{-\mez}(\T^d), H^\mez(\T^d)}\ge M\frac{p-1}{p} \| |g|^{p/2-1}g\|^2_{\dot H^\mez(\T^d)}\ge M\mez \| |g|^{p/2-1}g\|^2_{\dot H^\mez(\T^d)},
\]
where $M$ is given by \eqref{def:M}. It then suffices to prove that  for some $C'=C'(p, d)>0$,
\bq\label{Pioncare:Lp}
\| g\|_{L^p(\T^d)}\le C'\| |g|^{p/2-1}g\|_{\dot H^\mez(\T^d)}^\frac{2}{p}+C'\int_{\T^d}g.
\eq
For the sake of contradiction, assume that for all $n\in \Nn$, there exists $g_n\ne 0$ such that
\bq\label{Pioncare:2}
\frac{1}{n}\| g_n\|_{L^p(\T^d)}\ge \| |g_n|^{p/2-1}g_n\|_{\dot H^\mez(\T^d)}^\frac{2}{p}+\int_{\T^d}g_n.
\eq
By the homogeneity of \eqref{Pioncare:2} in $g_n$, we can assume that $\| g_n\|_{L^p(\T^d)}=1$ for all $n$. Set  $q_n=|g_n|^{p/2-1}g_n$. We have $\|q_n\|_{L^2}=\| g_n\|_{L^p}^{p/2}=1$ and thus the sequence $(q_n)$ is bounded in $H^\mez(\T^d)$. By the compact embedding $H^\mez(\T^d)\subset L^2(\T^d)$, there exists a subsequence, which we renumber $(q_n)$, that converges weakly to $q$ in $H^\mez(\T^d)$ and converges strongly  to $q$ in $L^2(\T^d)$. In particular, we have $\| q\|_{L^2}=1$. On the other hand, \eqref{Pioncare:2} implies that $\| q_n\|_{\dot H^\mez}\le 1/n$, whence $\| q\|_{\dot H^\mez}=0$ and hence $q=c$ is a constant. Since $\|q\|_{L^2}=1$, $c$ must be nonzero. Assume without loss of generality that $c>0$. From \eqref{Pioncare:2} we deduce 
\[
0=\lim_{n\to \infty}\int_{\T^d}g_n(x)dx=\lim_{n\to \infty}\int_{\T^d} |q_n(x)|^{2/p}\text{sign}(q_n(x))dx.
\]
Since $q_n\to q=c$ in $L^2$, there exists a subsequence, which we renumber $q_n$, such that $q_n(x)\to c$ a.e. $\T^d$ and there exists $Q\in L^2(\T^d)$ such that  for all $n$, $|q_n(x)|\le Q(x)$ a.e. $\T^d$.  Then $|q_n|^{2/p}\text{sign}(q_n)\to c^{2/p}$ and $||q_n|^{2/p}\text{sign}(q_n)|\le |Q|^{2/p}$ a.e. $\T^d$. Since $Q\in L^2(\T^d)$, we have $|Q|^{2/p}\in L^p(\T^d)\subset L^1(\T^d)$ for all $p\ge 1$.  Therefore, the dominated convergence theorem yields 
\[
0=\int_{\T^d} c^{2/p}=c^{2/p}|\T^d|.
\]
This contradicts the fact that $c>0$.
\end{proof}
\begin{lemm}\label{lemm:chainrule}
Let $U\subset \Rr^N$ be an open set and let $\Gamma: \Rr\to \Rr$ be a $C^1$ function. If  $u\in \Ll^\infty(U)$ and $\na u\in \Ll^1(U)$, then $\na \Gamma(u)=\Gamma'(u)\na u$. 
\end{lemm}
\begin{proof}
Let $V\Subset W \Subset  U$. Let $\rho_n$ be the standard mollifier at scale $1/n$ and set $u_n=(u 1_W)*\rho_n$, where $1_V$ is the indicator function of $V$. Since $u\in \Ll^\infty(U)$ and $\na u\in \Ll^1(U)$, we have that  
\begin{align*}
&\na u_n\to \na u\quad\text{in } W^{1, 1}(V),\\
&\exists M>0,~\forall n,~\| u_n\|_{L^\infty(\Rr^N)}+\| u\|_{L^\infty(U)}\le M.
\end{align*} 
It follows that
\[
\int_V |\Gamma(u_n)-\Gamma(u)|\le \max_{[-M, M]}|\Gamma'|\int_V |u_n-u|\to 0\quad\text{as~}n\to \infty
\]
and
\begin{align*}
\int_V|\Gamma'(u_n)\na u_n-\Gamma'(u)\na u|&\le \int_V |\Gamma'(u_n)||\na u_n-\na u|+\int_V|\Gamma'(u_n)-\Gamma'(u)||\na u|\\
& \le \max_{[-M, M]}|\Gamma'|\int_V |\na u_n-\na u|+\int_V|\Gamma'(u_n)-\Gamma'(u)||\na u|.
\end{align*}
A subsequence of $(u_n)$, which we renumber $(u_n)$, must converge a.e. to $u$ in $V$. Hence the last integral converges to $0$ by the dominated convergence theorem. Consequently the sequences $(\Gamma(u_m))$, $(\Gamma'(u_n)\na u_n)$ converge to $\Gamma(u)$, $\Gamma'(u)\na u$ respectively in $L^1(V)$. Since $\na \Gamma(u_n)=\Gamma'(u_n)\na u_n$ and $V$ is arbitrary,  we conclude that  $\na \Gamma(u)=\Gamma'(u)\na u$. 
\end{proof}

\section{Time decay for the one-phase Muskat problem}\label{section:Muskat}
The one-phase Muskat problem concerns the dynamics of the free boundary of a fluid occupying a region in a porous medium. The fluid motion is modeled by Darcy's law with gravity. When the fluid domain has the form \eqref{striplike} or \eqref{halfspace}, the free boundary $f$ obeys the equation
\bq\label{eq:Muskat}
\p_t f=-G_f(f),
\eq
where we write $G_f$ to emphasize the dependence of $G$ on the free boundary $f$. Some physical constants have been normalized in \eqref{eq:Muskat}. We refer to \cite{NgPa} for a derivation of \eqref{eq:Muskat}. 

We recall the following global well-posedness result.
\begin{theo}[\protect{\cite[Theorem 1.2]{DGN}}]
Let $\Omega$ be the domain \eqref{halfspace} with $M=\T$. For any initial data $f_0\in W^{1, \infty}(\T)$, equation \eqref{eq:Muskat} has a unique viscosity solution 
\bq\label{reg:f}
f\in C(\T\times [0, \infty))\cap L^\infty([0, \infty); W^{1, \infty}(\T)),\quad \p_t f\in L^\infty([0, \infty); L^2(\T)).
\eq
In particular, \eqref{eq:Muskat} is satisfied in the $L^\infty_t L^2_x$ sense.  Moreover, we have
\bq\label{max:f}
\| f(t)\|_{L^\infty(\T)}\le \|f(0)\|_{L^\infty(\T)},\quad\int_\T f(x, t)dx=\int_\T f(x, 0)dx\quad\forall t>0
\eq
and
\bq\label{max:df}
 \| \p_xf(t)\|_{L^\infty(\T)}\le \|\p_xf(0)\|_{L^\infty(\T)}\quad a.e.~ t>0.
\eq
\end{theo}
The precise definition of viscosity solutions of \eqref{eq:Muskat} is given in Definition 6.1 in \cite{DGN}. In what follows, we will only need the fact that equation \eqref{eq:Muskat} is satisfied in $L^\infty_t L^2_x$. 

We now apply the coercive estimates in the preceding section to prove the following result on time decay of the solutions.
\begin{prop}
For any $f_0\in W^{1, \infty}(\T))$, we have 
\[
f\in L^2([0, \infty); \dot H^\mez(\T)),\quad \p_tf\in L^2([0, \infty); H^{-\mez}(\T)).
\]
If in addition $\int_{\T} f_0=0$, then $\| f(t)\|_{H^\alpha}$, $\| f(t)\|_{C^\alpha}$ and $\| \p_t f(t)\|_{H^{-\eps}}$ decay exponentially as $t\to \infty$ for any $\alpha \in (0, 1)$ and any $\eps>0$.
\end{prop}
\begin{proof}
Thanks to the regularity \eqref{reg:f}, the following calculation is justified:
\bq
\mez\frac{d}{dt}\int_\T f^2(x, t)dx=(\p_t f(t), f(t))_{L^2, L^2}=\langle\p_t f(t), f(t)\rangle_{H^{-\mez}, H^\mez}=-\langle G_f(f), f\rangle_{H^{-\mez}, H^\mez}.
\eq
Applying Proposition \ref{prop:coercive:inf} and the maximum principle \eqref{max:df}, we deduce
\bq\label{dt:L2}
\mez\frac{d}{dt}\| f(t)\|^2_{L^2(\T)}\le -\frac{C}{1+\| \p_x f(t)\|_{L^\infty(\T)}}\|f(t)\|_{\dot H^\mez(\T)}^2\le  -\frac{C}{1+\| \p_x f(0)\|_{L^\infty(\T)}}\|f(t)\|_{\dot H^\mez(\T)}^2
\eq
for a.e. $t>0$.  It follows that
\bq\label{L2Hmez}
f\in L^2([0, \infty); \dot H^\mez(\T)).
\eq
Combining \eqref{L2Hmez} and \eqref{bound:DNinf} yields
\bq\label{L2dtf}
\p_tf=-G_f(f)\in L^2([0, \infty); H^{-\mez}(\T)).
\eq
Assume now that $f_0$ has zero mean, then \eqref{max:f} implies that $f(t)$ has zero mean for all $t>0$. Consequently, $\| f(t)\|_{\dot H^\mez}\ge \| f(t)\|_{L^2}$ and thus \eqref{dt:L2}
yields
\bq
\frac{d}{dt}\| f(t)\|^2_{L^2(\T)}\le -\frac{C}{1+\| \p_x f(0)\|_{L^\infty(\T)}}\|f(t)\|_{L^2(\T)}^2\quad\forall t>0.
\eq
Therefore, the $L^2$ norm of $f$ decays exponentially,
\bq\label{decay:L2}
\| f(t)\|_{L^2}\le \| f_0\|_{L^2}e^{-\frac{Ct}{1+\| \p_x f(0)\|_{L^\infty(\T)}}}\quad\forall t>0.
\eq
Combining \eqref{decay:L2} with the uniform  bounds \eqref{max:f} and \eqref{max:df}, we deduce that $f$ decays exponentially in any norms that interpolate between $L^2(\T)$ and $W^{1, \infty}(\T)$. In particular, all the $H^\alpha(\T)$ and $C^\alpha(\T)$ norms, $\alpha\in [0, 1)$, of $f$ decay exponentially. Next, we recall from \cite{DGN} that 
\bq
\| G_f(f)\|_{H^{\sigma-1}}\le C(1+\| \p_x f\|_{L^\infty})^2\| f\|_{\dot H^\sigma},\quad \sigma\in [\mez, 1]. 
\eq
Therefore, $\p_t f=-G_f(f)$ decays exponentially in $H^{-\eps}(\T)$ for any $\eps>0$. 
\end{proof}

\vspace{.1in}
\noindent{\bf{Acknowledgment.}} 
The work of HQN was partially supported by NSF grant DMS-2205734. The author thanks H. Dong and F. Gancedo for stimulating and helpful discussions. He also thanks the referees for useful suggestions. 
 

\begin{thebibliography}{10}
\small
\bibitem{ABZ3}
T. Alazard, N. Burq, and C. Zuily.
\newblock On the Cauchy problem for gravity water waves.
\newblock {\em Invent. Math.},  198 (2014), no. 1, 71--163.

\bibitem{AMS} T. Alazard, N. Meunier, D. Smets. Lyapunov functions, identities and the Cauchy problem for the Hele-Shaw equation..  {\em Comm. Math. Phys.} 377 (2020), no. 2, 1421--1459.

 \bibitem{ChlGuiSch} H. A. Chang-Lara, N. Guillen, R. W. Schwab. Some free boundary problems recast as nonlocal parabolic equations. {\em Nonlinear Anal.} 189 (2019), 11538, 60 pp.

\bibitem{CC0} A. C\'ordoba, D. C\'ordoba. A pointwise estimate for fractionary derivatives with applications to partial differential equations. {\em Proc. Natl. Acad. Sci. USA} 100 (2003), no. 26, 15316--15317.
 
\bibitem{CC} A. C\'ordoba, D. C\'ordoba. A maximum principle applied to quasi-geostrophic equations. {\em Comm. Math. Phys.} 249 (2004), no. 3, 511--528.
 
 
 \bibitem{DGN} H. Dong, F. Gancedo, H. Q. Nguyen. Global well-posedness for the one-phase Muskat problem. arXiv:2103.02656, 2021, to appear in {\em Comm. Pure Appl. Math.}. 

\bibitem{LeoniTice} G. Leoni, I. Tice. Traces for homogeneous Sobolev spaces in infinite strip-like domains. {\em J. Funct. Anal.} 277 (2019), no. 7, 2288--2380.

\bibitem{NgPa} Huy Q. Nguyen, B. Pausader. A paradifferential approach for well-posedness of the Muskat problem. {\em Arch. Ration. Mech. Anal.} 237 (2020), no. 1, 35--100.
 
\end{thebibliography}
\end{document}